\documentclass{dcds}
\usepackage{relsize}
\usepackage{amsmath}
  \usepackage{paralist}
 \usepackage{psfrag, subfigure}
  \usepackage{graphics} 
  \usepackage{epsfig} 
\usepackage{graphicx}  \usepackage{epstopdf}
 \usepackage[colorlinks=true]{hyperref}
\hypersetup{urlcolor=blue, citecolor=red}

\def\currentvolume{}
 \def\currentissue{}
  \def\currentyear{}
   \def\currentmonth{}
    \def\ppages{X--XX}
     \def\DOI{}

\newtheorem{theorem}{Theorem}[section]
\newtheorem{coro}{Corollary}

\newtheorem{lemma}[theorem]{Lemma}

\theoremstyle{definition}

\newtheorem{rk}{Remark}

\newcommand{\N}{\mbox{$\mathbb{N}$}}
\newcommand{\Z}{\mbox{$Z\!\!\!Z$}}

\newcommand{\R}{\mbox{$I\!\!R$}}

\numberwithin{equation}{section}

\begin{document}

\title[ Sensitivity of semigroups ]{ Shadowing Property for the free group acting in the circle}

\author[Jorge Iglesias and Aldo Portela]{}

\subjclass{Primary: 37B05; Secondary:  37C85 (37C50 37D20).}
 \keywords{Group acting,shadowing, minimality.}

%
\email{jorgei@fing.edu.uy }

\email{aldo@fing.edu.uy }

\maketitle

\centerline{\scshape  Jorge Iglesias$^*$ and Aldo Portela$^*$}
\medskip
{\footnotesize
 \centerline{Universidad de La Rep\'ublica. Facultad de Ingenieria. IMERL}
   \centerline{ Julio Herrera y Reissig 565. C.P. 11300}
   \centerline{ Montevideo, Uruguay}}

\bigskip

 \centerline{(Communicated by )}

\begin{abstract}
For the free group $F_2$ acting in $S^{1}$, we will prove that if the minimal set for the action is not a Cantor set, then the action does not have the shadowing property.
We will also construct an example, whose minimal set is a Cantor set, that it has the shadowing property.
\end{abstract}

\section{Introduction.}
The concepts of pseudo-orbits and shadowing property for homeomorphisms were widely investigated by many authors.
Many results are known that link expansivity and hyperbolicity, with the shadowing property. In \cite{p} it is find a survey of the most important results. In \cite{ot} this concept was generalized for finitely generated groups acting in a metric space $X$. In the said article, conditions are given that imply that an action does not have the shadowing property and examples with the shadowing property are constructed.  It is also conjectured that the action of a free group of finitely generators, on a manifold $M$, can not have the shadowing property.
When the manifold is $S^{1}$, the minimal sets are classified, being all $S^{1}$, a finite set or a Cantor set. The objective of this paper is to show that when the minimal set is all $S^{1}$ or a finite set, the action can not have shadowing. We will construct an example (whose minimal set is a Cantor set) that has the shadowing property; proving that what was conjectured is false.

\subsection{Basic definition.}
Given a group $G$ and a set $X$,  a dynamical system is formally define as a triplet $(G,X,\Phi )$, where $\Phi:G \times X \to X$ is a continuous function with $\Phi (  g_1, \Phi(g_2 ,x))= \Phi(g_1 g_2 ,x)$ for all $g_1,g_2 \in G$ and for all $x \in X$. The map $\Phi$ is called an action of $G$ on $X$. Without loss of generality it is possible to associate each element of $G$ to a homeomorphism $\Phi_g :X \to X$.
For every $x \in X$ we define the orbit of $x$ as $O(x)=\{\Phi_g(x): \ g\in G\} $.\\

A group $G$ is finitely generated if there exists a finite set $S\subset G$ such that for any $g\in G$ there exist $s_1,..., s_n\in S$ with $g=s_1.\cdots . s_n$.
The set $S$ is called finite generator of $G$. If $S$ is a finite generator of $G$ and for all $s\in S$ we have that $s^{-1}\in S$, then the set $S$ is called a finite symmetric generator.\\
For usual dynamical systems, this is when the group is  $\Z$ and the action is $\Phi (x,n)=f^{n}(x)$, we say that a sequence $\{x_n\}$ is
$\delta$-pseudotrajectory if $$d(f(x_n), x_{n+1})<\delta ,  \ \ \forall n\in Z.$$ It is possible to generalize this concept to a dynamical system $(G,X,\Phi )$, as follows:
A $G$-sequence in $X$ is a function $F:G\to X$. We denote this function by $\{ x_g  \} $ where $F(g)=x_g$.
Let $S$ be a finite symmetric generator of $G$. For $\delta >0$ we say that a $G$-sequence $\{ x_g  \} $  is a $\delta$-pseudotrajectory
 if
$$d(\Phi_s(x_g), x_{sg})<\delta ,\ \  \forall g\in G \mbox{ and } \forall s\in S.$$

Given a dynamical system $(G,X,\Phi)$, $\Phi$  has the shadowing
property if for any $\varepsilon > 0$ there exists $\delta > 0$ such that
for any $\delta$-pseudotrajectory $\{ x_g  \} $  there exists a point $y \in X$ with
$$d(x_g, \Phi_g(y)) < \varepsilon, \ \forall g \in G.$$

A finite symmetric generator $S$ of $G$ is uniformly continuous if for every $\varepsilon >$ 0 there exists
$\delta  > 0$ such that $d(x, y) < \delta$ implies $d(\Phi_s (x),\Phi_s (y))<\varepsilon $ for every $s\in S$.
In {\cite[Proposition 1]{ot}} it its proved that the shadowing property does not depend on the finite symmetric generator $S$ when is uniformly continuous.\\

Given a dynamical systems $(G,X,\Phi)$, we say that a map $\Phi_g$ is expansive if there exists $\alpha  >0$ such that if

$$d(\Phi_{g^{n}}(x), \Phi_{g^{n}}(y)) < \alpha , \ \forall n \in \Z   \mbox{  then } x=y  .$$

In  {\cite[Theorem 4]{ot}} it was proved that:

\begin{theorem}\label{teoot}

Let $(G,X,\Phi)$ be a dynamical system where $G$ be a finitely generated free group with at least two generators,
 $\Phi$ is uniformly continuous and $X$ is a non-discrete
metric space.
\begin{enumerate}
\item If for some $g \in  G$ the map $\Phi_g$ is expansive, then $\Phi$ does not have
shadowing.
\item If for some $g \in G$, $g \neq e$, the map $\Phi_g$ does not have shadowing, then $\Phi$
does not have shadowing either.
\end{enumerate}
\end{theorem}

 We say that a point $N\in X$ is an expansive point for $g\in G$ if there exists $\alpha >0$ such that for any $\delta >0$ if $0<d(y,N)<\delta$, then there exists $n\in \Z$ such $d(\Phi_{g^{n}}(y), \Phi_{g^{n}}(N))>\alpha$. The number $\alpha$ is called expansivity constant.
In this paper we will prove the following result which in some cases is a generalization of the item 1 above.\\
{\bf{ Lemma.}}{\it{ Let $(G,X,\Phi)$ be a dynamical system where $G$ be a finitely generated free group with at least two generators, $\Phi$ is uniformly continuous, $X$ is a metric space and $N$ is a non isolated  expansive point for $\Phi_g$ for some $g\in G$. If there exists a connected and invariant set $M$ for the action $\Phi$ with $M\subset \overline{O(N)}$, then $\Phi$ does not have the shadowing property.}}\\

A non-empty set $M$, $M \subset X$, is minimal if $\overline{O(x)}=M$ for any $x \in M$.
When $X=S^{1}$ we have the following result (see for example \cite{n}):\\
If $M\subset S^{1}$ is a minimal set, then
one of these three possibilities occurs:
\begin{enumerate}
\item is a finite orbit of $\Phi$,
\item is $ S^1$,
\item  is a Cantor set and is the unique minimal set for $\Phi$.
\end{enumerate}
Now we consider the free group of two generators that we will denote by $F_2$.

Let us state our main results:\\
{\bf{Theorem A}}
{\it{ Let $(F_2,S^{1},\Phi)$ be a dynamical system. If M is a minimal set which is not a Cantor set, then $\Phi$ does not have the shadowing property.}}

When $M$ is a Cantor set we construct an example which it has the shadowing property. This example is easily generalizable to $S^{n}$.

\section{Construction of the example}

In this section, we considerer the free group $F_2$ with finite symmetric generator $S=\{a,a^{-1},b,b^{-1}\}$.
We are going to construct an action $\Phi$ in $S^{1}$ whose minimal set $K$ is a Cantor and has shadowing property. The generator of the action will be $\Phi_a $ and $\Phi_b$ where $\Phi_a, \Phi_b :S^{1}\to S^{1}$ have the following properties (see figure \ref{figura101}):\\
\begin{enumerate}
\item   $\Phi_a, \Phi_b$ are north-south pole homeomorphisms, with $\Omega (\Phi_a)=\{ N_a,S_a  \}$, $\Omega (\Phi_b)=\{ N_b,S_b  \}$.\\
\item For $\Phi_a$ there exist two open balls $B_{N_{a}}=B(N_{a},r_1)$ y $B_{S_{a}}=B(S_{a},r_2)$ such that:
\begin{itemize}
\item   $\overline{B_{N_{a}}}\cap \overline{  B_{S_{a}}}=\emptyset$.\\
\item  $||\Phi_a(x)-\Phi_a(y)||>2|| x-y   ||$ for all $x,y\in B_{N_{a}} $ and\\ $||\Phi_a(x)-\Phi_a(y)||<1/2|| x-y   ||$ for all $x,y\in \overline{B_{N_{a}}}^{c}$.\\

\item $\Phi_a (\partial B_{N_{a}})\subset B_{S_{a}}$.
\end{itemize}

\item For $\Phi_b$ there exist two balls  $B_{N_{b}}=B(N_{b},r_3)$ y $B_{N_{b}}=B(S_{b},r_4)$  with the same properties given in item 2, and with the additional condition

$\overline {B_{N_{b}}  \cup B_{S_{b}}} \subset ( \overline {B_{N_{a}}  \cup B_{S_{a}}})^{c}.$\\

\end{enumerate}

\begin{figure}[ht]
\begin{center}
\caption{\label{figura101}}
\psfrag{phia}{$\Phi_a$}
\psfrag{bna}{$B_{N_a}$}
\psfrag{sa}{$S_a$}
\psfrag{ib}{$I_b$}
\psfrag{ibb}{$I_{b^{-1}}$}
\psfrag{na}{$N_a$}

\includegraphics[scale=.2]{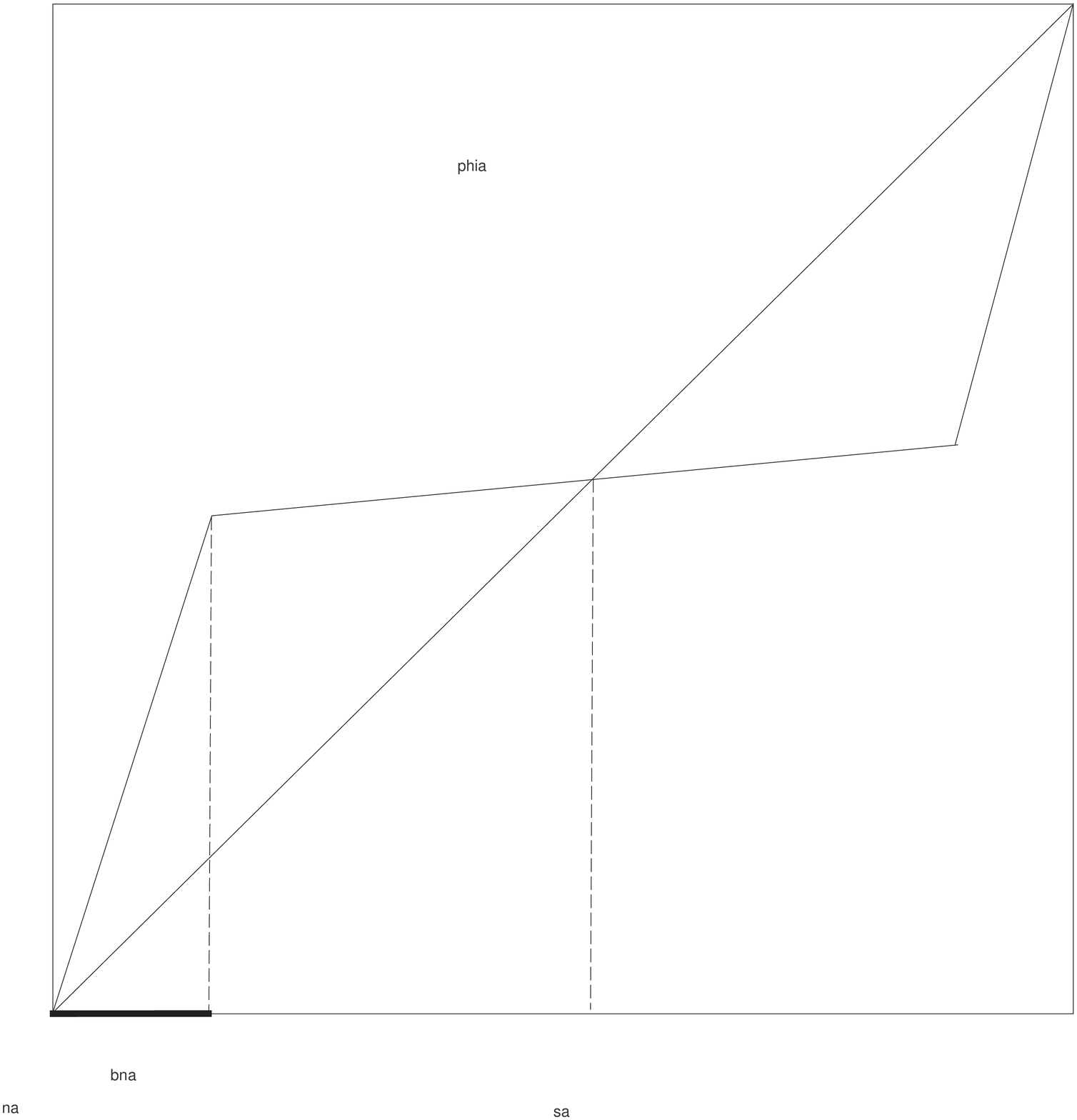}
\end{center}
\end{figure}

Let $I_a=(\Phi_a( B_{N_{a}}  ))^{c}$, $I_{a^{-1}}= (\Phi^{-1}_a( B_{S_{a}}  ))^{c}$, $I_b=\Phi_b(( B_{N_{b}}  ))^{c}$  and  $I_{b^{-1}}=(\Phi^{-1}_b( B_{S_{b}}  ))^{c} $


 The following properties are very useful for our purpose. Since they are not hard to prove we omit its proof.

\begin{rk}\label{rk1}
\begin{enumerate}
\item If $s,s^{'}\in \{ a,a^{-1},b,b^{-1}\}$ then $$||\Phi_s(x)-\Phi_s(y)||<1/2|| x-y   || \ \forall x,y\in I_{s^{'}} \mbox{ with } s^{'}\neq s^{-1}.$$
\item If $s\neq s{^{'}{^{-1}}}$ then $\Phi_s (I_{s^{'}})\subset I_s$.
\item $\Phi_s(I_s)\cap I_{s^{-1}}=\emptyset$ for all $s\in \{a,b,a^{-1},b^{-1}\}$.\\
\end{enumerate}
\end{rk}

Let $\{A_n\}$ be such that
 $$A_0=I_a\cup I_{a^{-1}}\cup I_b\cup I_{b^{-1}} \mbox{ and }$$ $$  A_{n+1}= \left[ \Phi_a(A_n)\cap \Phi_{a^{-1}}(A_n)\cap \Phi_b(A_n)\cap \Phi_{b^{-1}}(A_n)   \right] \cap A_n.$$
 Note that (see figure \ref{figura1})

 \begin{itemize}
\item For any $n\in\N$, $A_n$ has $4.3^{n}$ connected components and $A_{n+1}\subset int(A_n)$.
\item The lengths of the connected components fo $A_n$ goes to zero when $n$ goes to infinity.
\item The Cantor set  $K=\bigcap_{n\geq 1}A_n$ is a minimal for the action $\Phi$ generated for $\Phi_a$ y $\Phi_b$ (see \cite{n}).
\end{itemize}

Some of our proofs are by induction in the length of the elements $ g\in G $. Thus we need to define the length of an element $g\in G $.

The elements of length one are $ a,a^{-1},b$ and $b^{-1}$. The elements of length $n$ are obtained from the elements of length $n-1$ as follows: Let $g=s_{1}....s_{n-1}$ be an element of lengths $n-1$ with $s_j\in \{ a,a^{-1},b,b^{-1}\}$. Then the element of length $n$ generated by $g$ are $g^{'}=s.g$ with $s\neq  (s_{1})^{-1}$. The length of $g$ is denoted by $|g|$.

It is clear that an element $g \in G $ can be written from $ S $ in different ways, for example $ g = gaa ^ {- 1} $.
Note that if $g=s_{{1}}....s_{{n}}$ with $s_j\in \{ a,a^{-1},b,b^{-1}\}$, then  $|g|\leq n$. We say that $g=s_{{1}}....s_{{n}}$ written in its normal form if $|g|= n$. It is easy to prove that the normal representation is unique.

From now we will consider $g\in G$ written in its normal form.

\begin{lemma}\label{lema_induccion}
Let $ x\notin A_0$. If $n\geq 2$, then for any $g\in G$ with $2\leq |g|\leq n$ hold:
\begin{itemize}
\item $\Phi_g (x) \in A_0$
\item If $\Phi_g (x) \in I_{s}$ with $s\in \{ a,a^{-1},b,b^{-1}\}$ and $g=s_{1}\cdots s_{n}$, then $s_{1}=s$.
\end{itemize}
  \end{lemma}
\begin{proof}
The proof is by induction in the length of $g$.\\
If $|g|=2$ or 3, analyzing all possible cases for $g$ the thesis is fulfilled.\\
Suppose that the statement is valid for all $g\in G$ with $|g|\leq n$. Let $g^{'}\in G$ be with $|g^{'}|=n+1$.
Thus $g^{'}=sg$ with $s\in \{ a,a^{-1},b,b^{-1}\}$ and $|g|=n$.

Since $|g|=n$, by the assumptions, $\Phi_{g} (x) \in A_0$. If $\Phi_{g} (x) \in I_a$ (the other cases are analogous), then by the second part of assumptions, we obtain that  $g=ag^{''}$. Thus $s\neq a^{-1}$, hence $\Phi_s\neq \Phi_{a^{-1}}$. Since $\Phi_{g} (x) \in I_a$ and $s\in \{ a,b,b^{-1}\}$, by Remark \ref{rk1} item 2), $\Phi_s\Phi_{g} (x)=\Phi_{g^{'}} (x)\in I_s$ and the thesis is verified.\\
\end{proof}

\begin{coro}\label{lema_adentro}
 For any $x\in S^{1}$ we have that  $\sharp O(x)\cap A_0^{c}\leq 2$.
\end{coro}

\begin{proof}
Suppose that there exists $z\in O(x)$ with $z\in A_0^{c}$. Thus, by lemma above, if $|g|\geq 2$, then $\Phi_g(z)\in A_0$.
When $|g|\leq 1$ analyzing the different cases it is easy to show that there are at most two point of $O(z)$ in $A_0^{c}$.
Thus, Since $O(z)=O(x)$ we obtain the thesis.
\end{proof}

\begin{figure}[h]
\psfrag{i0}{$I_{b^{-1}}$}

\psfrag{i1}{$I_{b}$}

\psfrag{j0}{$I_{a^{-1}}$}

\psfrag{j1}{$I_{a}$}

\psfrag{fi0}{$\Phi_a(I_{b^{-1}})$}

\psfrag{fi1}{$\Phi_a(I_{b})$}

\psfrag{fj0}{$\Phi_a(\partial I_{a^{-1}})$}

\psfrag{fj1}{$\Phi_a(I_{a})$}

\psfrag{gi0}{$\Phi_b(I_{a})$}

\psfrag{gi1}{$\Phi_b(I_{b})$}

\psfrag{gj0}{$\Phi_b(I_{a^{-1}})$}

\psfrag{gj1}{$\Phi_b(I_{b})$}

\psfrag{ggi0}{$\Phi_b^{-1}(I_{a})$}

\psfrag{ggi1}{$\Phi_b^{-1}(\partial I_{b})$}

\psfrag{ggj0}{$\Phi_b^{-1}(I_{a^{-1}})$}

\psfrag{ggj1}{$\Phi_b^{-1}(I_{b^{-1}})$}

\psfrag{fbi3}{$\Phi_b^{}(\partial I_{b^{-1}})$}

\psfrag{ffi0}{$\Phi_a^{-1}(\partial I_{b^{-1}})$}

\psfrag{ffi1}{$\Phi_a^{-1}(I_{b})$}

\psfrag{ffj0}{$\Phi_a^{-1}(I_{a^{-1}})$}

\psfrag{ffj1}{$\Phi_a^{-1}(I_{b^{-1}})$}

\psfrag{ffi0}{$\Phi_a^{-1}(I_{b^{-1}})$}

\psfrag{ffi1}{$\Phi_a^{-1}(I_{b})$}

\psfrag{ffj0}{$\Phi_a^{-1}(I_{a^{-1}})$}

\psfrag{ffj1}{$\Phi_a^{-1}(I_{b^{-1}})$}

\begin{center}
\caption{\label{figura1}}
\subfigure[]{\includegraphics[scale=0.27]{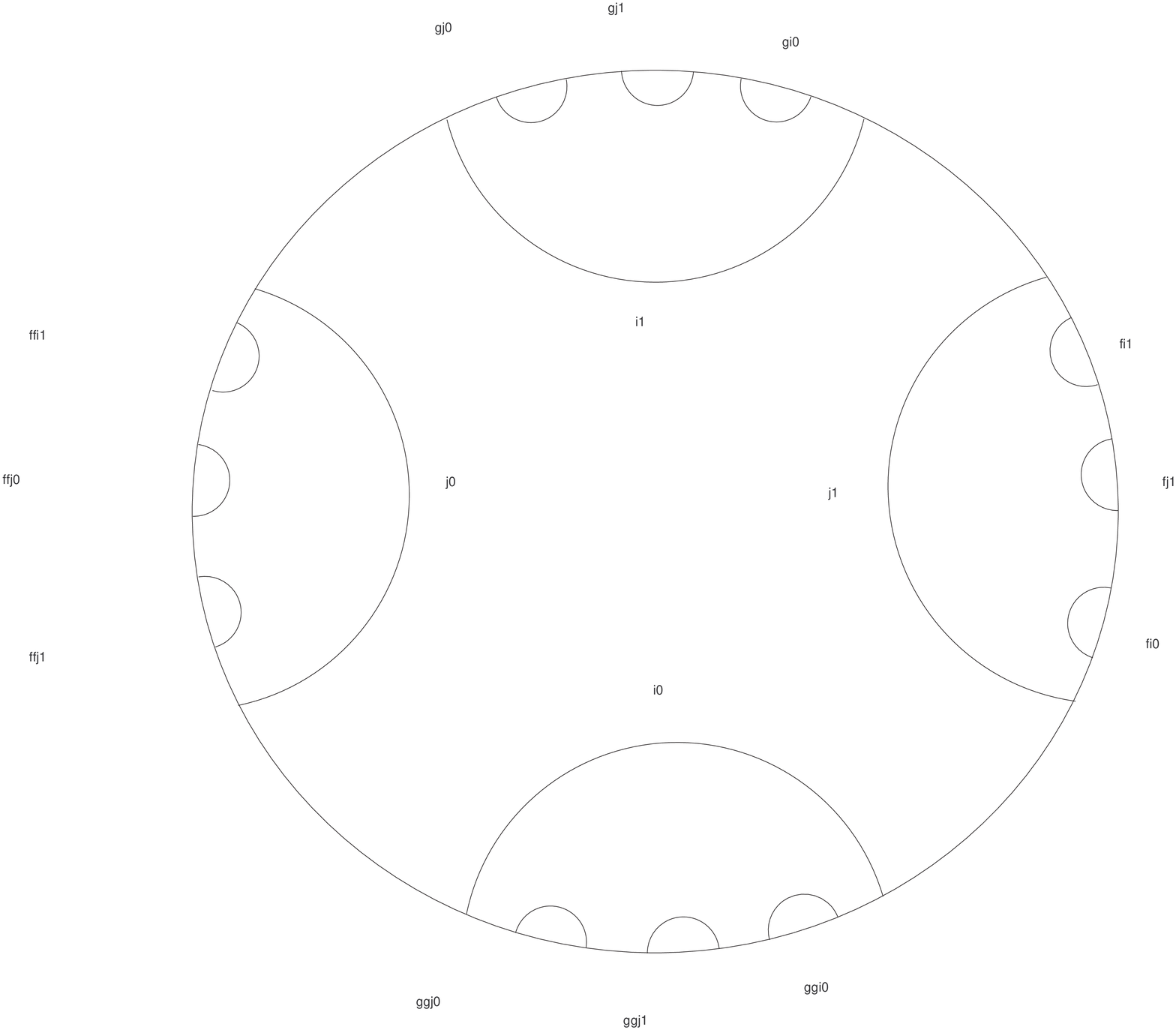}}
\end{center}
{ These figures correspond to example 1.}
\end{figure}

Now we will prove that this example has the shadowing property. We divide the proof in two cases. When the pseudotrajectory is contained in $A_0$ and when it is not contained in  $A_0$.

{\bf{Pseudotrajectory contained in $A_0$.}} Given a pseudotrajectory $\{x_g\}$, $\{x_g\}\subset A_0$,  for any $g\in G$ there exists $s\in \{a,a^{-1},b,b^{-1} \}$ such that  $x_g\in I_s$.

We want to find a point $y\in K$ such that $x_g\in I_s$  iff  $\Phi_g(y)\in I_s$ for all $g\in G$.

\begin{lemma}\label{l1}
 There exists $\delta_0>0$ such that for any $\delta$, with $0<\delta <\delta_0$, and $\{ x_g  \} $ a $\delta$-pseudotrajectory contained in $ A_0$, then:
\begin{enumerate}
  \item  If $x_g\in I_s$, then $x_{s^{-1}g}\notin I_{s^{-1}}$, for $s\in \{a,b,a^{-1},b^{-1}\}$.
  \item  If $\Phi_s(x_g)\in I_{s^{'}}, $ for $s,s^{'}\in \{a,b,a^{-1},b^{-1}\}$, then $x_{sg}\in I_{s^{'}}$.
\end{enumerate}
\end{lemma}
\begin{proof}
  1. By Remark \ref{rk1} item 3, we have that $\Phi_s(I_s)\cap I_{s^{-1}}=\emptyset$, for all $s\in \{a,b,a^{-1},b^{-1}\}$. From the uniform continuity of the maps $\Phi_s$ and
 $d(x_{s^{-1}g}, \Phi_{s^{-1}}(x_g))< \delta $, there exists $\delta_1 >0$ such that if $\delta <\delta_1$ and $x_g\in I_s$ then $x_{s^{-1}g}\notin I_{s^{-1}}$. \\

  2. Let $J_0,J_1,J_2$ and $J_3$ be the connected components of $S^{1}\setminus A_0$. Let $\rho=min \{ |J_0|,|J_1|,|J_2|, |J_3 |\}$ where $|J_i|$ is the lengths of the interval $J_i$.
 Since $\Phi_s(x_g), x_{sg} \in A_0$ and $d(x_{s^{}g}, \Phi_{s^{}}(x_g))< \delta $,  if $\delta <\rho$ then  $\Phi_s(x_g), x_{sg}\in I_{s^{'}} $ for $s,s^{'}\in \{a,b,a^{-1},b^{-1}\}$.

 Hence, taking $\delta_0 = min \{\delta_1, \rho \}$ the thesis is verified.

\end{proof}


\begin{lemma}\label{l2}
There exists  $\delta_0>0$ such that,  if $\{ x_g  \} $ is a $\delta$-pseudotrajectory with $0<\delta <\delta_0$ and $\{ x_g  \} \subset A_0$, then there exists $y\in K$ such that $ x_g\in I_s$ iff $\Phi_g (y)\in I_s$.
\end{lemma}

\begin{proof}

Let $\delta_0$ be given by Lemma \ref{l1} and $\delta >0$ with $0<\delta <\delta_0$.

Let $s_0\in \{a,b,a^{-1},b^{-1}\}$ be such that $x_e\in I_{s_{0}}$. Thus, we define $T_0=I_{s_{0}}$.

Now we consider the point $x_{s^{-1}_{0}}$.\\

Let $s_1\in \{a,b,a^{-1},b^{-1}\}$ be such that $x_{s^{-1}_{0}}\in I_{s_{1}}$.

  Since $\delta <\delta_0$, by Lemma \ref{l1} item 1), $x_{s^{-1}_{0}}\notin I_{s^{-1}_{0}}.$ Hence $s_{1} \neq s^{-1}_{0}$.

  Since $s_{1} \neq s^{-1}_{0}$, by Remark \ref{rk1}, item 2,)  $\Phi^{-1}_{s^{-1}_{0}} ( I_{s_{1}}) =\Phi_{s_{0}} ( I_{s_{1}})  \subset I_{s_{0}} $.

 Thus, we define $T_1= \Phi^{-1}_{s^{-1}_{0}} ( I_{s_{1}})\subset T_0$.

  Consider the point $x_{    s^{-1}_{1}    s^{-1}_{0}}$ .
Let $s_2\in \{a,b,a^{-1},b^{-1}\}$ be such that $x_{    s^{-1}_{1}    s^{-1}_{0}}\in I_{s_{2}}$.

as above we have $s_{2} \neq s^{-1}_{1}$ and $\Phi^{-1}_{s^{-1}_{1}}   (I_{s_{2}})\subset I_{s_{1}} $. Thus, define $T_2= \Phi^{-1}_{s^{-1}_{0}}  \Phi^{-1}_{s^{-1}_{1}}  (I_{s_{2}})\subset T_1 \subset T_0.$

Now we construct inductively a sequence of nested intervals $\{T_n\}$. Since $|T_n|\to 0$, let $$\{y\}=\bigcap_{n\geq 0}T_n.$$

We will prove that $ x_g\in I_s$ iff $\Phi_g (y)\in I_s$.

Clearly $x_e$ and $\Phi_e (y)=y$ belong to $I_{s_{0}}$.

Let $s\in  \{a,b,a^{-1},b^{-1}\}$.

If $s\neq s^{-1}_0$, then by Remark \ref{rk1} item 2), $\Phi_s (I_{s_{0}})\subset I_s$. By Lemma \ref{l1} item 2, $x_s$ and $\Phi_s(y)$ belong to $I_s$.

If $s= s^{-1}_0$, by definition of $y$,  $x_{s^{-1}_0}$ and $\Phi_{s^{-1}_0} (y)$ belong to $I_{s_{1}}$.

By reasoning inductively the thesis is verified.

\end{proof}

\begin{lemma}\label{l3}
Let $\delta_0$ be given by Lemma \ref{l1}. If $\{ x_g  \}$ is a $\delta$-pseudotrajectory with $0<\delta <\delta_0$ and $\{ x_g  \} \subset A_0$, then there exists $y\in K$ such that
 $$d(x_g, \Phi_g (y))<3\delta \mbox{ for all  } g\in G .$$

\end{lemma}

\begin{proof}
Let $\{ x_g  \}$ be a $\delta$-pseudotrajectory . By Lemma \ref{l2} there exists $y\in K$  such that $ x_g\in I_s$ iff $\Phi_g (y)\in I_s$.\\
 let's prove that $d(x_g, \Phi_g (y))<3\delta $.
Suppose that there exists $g_0\in G$ such that $d(x_{g_{0}}, \Phi_{g_{0}} (y))\geq 3\delta $.  Let $I_s$ be such that $x_{g_{0}}\in I_s$ with $s\in \{ a,a^{-1},b,b^{-1}\}$. Since $ x_g\in I_s$ iff $\Phi_g (y)\in I_s$, then $\Phi_{g_{0}} (y)\in I_s$. Recall that  $||\Phi_{s^{-1}}(x)-\Phi_{s^{-1}}(y)||> 2|| x-y   ||$  for all $x,y\in I_{s}$.

Since $$d(x_{g_{0}}, \Phi_{g_{0}} (y))\geq 3\delta \mbox{ then }  d(\Phi_{s^{-1}}(x_{g_{0}}), \Phi_{s^{-1}} (\Phi_{g_{0}} (y)))= d(\Phi_{s^{-1}}(x_{g_{0}}), \Phi_{s^{-1}g_{0}} (y)) \geq 6\delta    $$

Also

$$d(\Phi_{s^{-1}}(x_{g_{0}}), \Phi_{s^{-1}g_{0}} (y))\leq d(\Phi_{s^{-1}}(x_{g_{0}}), x_{s^{-1}g_{0}} )+ d( x_{s^{-1}g_{0}}, \Phi_{s^{-1}g_{0}} (y)) $$
Since $d(\Phi_{s^{-1}}(x_{g_{0}}), x_{s^{-1}g_{0}} )<\delta $, we obtain

 $$ d(x_{s^{-1}g_{0}}, \Phi_{s^{-1}g_{0}} (y))   \geq 6\delta-\delta =5\delta .  $$

Thus, by reasoning inductively, there exists $g_n\in G$ such that  $d(x_{g_{n}}, \Phi_{g_{n}} (y))\geq n\delta $. This is a contradiction because $ x_g$ and $\Phi_g (y)$ belong to the same $I_s$.\\

\end{proof}

Given $\varepsilon >0$, taking  $\delta =\varepsilon /3$, we have the shadowing property for any $\delta$-pseudotrajectory contained in $A_0$.\\

{\bf{ Pseudotrajectory not contained in $A_0$.}}
\begin{lemma}
There exists $\delta_1>0$ such that if $\{ x_g  \} $ is a $\delta$-pseudotrajectory with $\delta <\delta_1$, then there exist at most two elements $g_{0},g_{1} \in G$ such that  $x_{g_{0}},x_{g_{1}} \notin A_0$.
  \end{lemma}
\begin{proof}
 By Lemma \ref{lema_adentro} for all  $x\in S^{1}$, $\sharp O(x)\cap A_0^{c}\leq 2$.
 Hence by uniform continuity of the maps $\Phi_s$ there exists $\delta_1>0$ that verifies the thesis.
\end{proof}

\begin{lemma}\label{lema_induccion2}

There exists $\delta_2>0$ such that  if $\{ x_g  \}$   is a $\delta$-pseudotrajectory with $\delta <\delta_2$ and $x_e\notin A_0$, then for any $g\in G$ with $2\leq |g|\leq n$ hold:
\begin{itemize}
\item $x_g \in A_0$,
\item If $x_g  \in I_{s}$, $s\in \{ a,a^{-1},b,b^{-1}\}$ and $g=s_{1}\cdots s_{n}$ then $s_{1}=s$.

\end{itemize}
  \end{lemma}

\begin{proof}
 By Lemma \ref{lema_induccion} the result is true for an orbit. Thus, by uniform continuity of maps $\Phi_s$ there exists $\delta_2>0$ such that the thesis is verified.\\
\end{proof}

Let $\{ x_g  \}$ be a $\delta$-pseudotrajectory.  Without loss of generality (renaming the pseudotrajectory if necessary) we can assume that $x_e\notin A_0$.\\

We will prove that for any $\varepsilon >0$ there exists $\delta >0$ such that, if $\{ x_g  \}$ is a $\delta$-pseudotrajectory, then $d(x_g,\Phi_g (x_e))<\varepsilon $ for all $g\in G$.

The proof is by induction in the length of $g$.\\

\begin{proof}

Let $\delta_0$ be given by Lemma \ref{lema_induccion2}. Given $\varepsilon >0$, let $\delta >0$ be such that:
\begin{enumerate}
  \item $\delta < min\{\delta_0, \varepsilon /2 \}$,
\item If $|g|\leq 2$, then
\begin{enumerate}
\item $d(x_g,\Phi_g (x_e))<\varepsilon  $ and
\item If $\Phi_g (x_e)\in I_s$, then $x_g\in I_s$, $s\in \{ a,a^{-1},b,b^{-1}\}$.
\end{enumerate}
\end{enumerate}

We will prove, by induction, that  $g\in G$, with $ |g|= n$, and $\Phi_g (x_e)\in I_s$ then $x_g\in I_s$  and $d(x_g,\Phi_g (x_e))<\varepsilon $.

 By item 2) above, we have the base case.

Let $g\in G$ be, $|g|=n+1> 2$. Thus  $g=sg^{'}$, with $|g^{'}|=n$. Since $x_e\notin A_0$ then, by Lemma \ref{lema_induccion2} item 2), $x_{g^{'}}\in I_{s^{'}}$ for some $s^{'}$ and $g^{'}=s^{'}g^{''}$. Since $|g|=n$ and $g=ss^{'}g^{''}$, then $s^{'}\neq s^{-1}$. Hence $||\Phi_s(x)-\Phi_s(y)||< 1/2|| x-y   ||$  for all $x,y\in I_{s^{'}}$.

 By assumption,  $\Phi_{g^{'}} (x_e)\in I_{s^{'}}$.  Hence

$$  d(x_g,\Phi_g (x_e))=d(x_{sg^{'}},\Phi_s\Phi_{g^{'}} (x_e))\leq  d(x_{sg^{'}},\Phi_s (x_{g^{'}}) )  +  d(  \Phi_s (x_{g^{'}} ), \Phi_s\Phi_{g^{'}} (x_e))       \leq          $$

$$   \delta + \frac{1}{2}     d(   (x_{g^{'}} ), \Phi_{g^{'}} (x_e))   .       $$

Since   $d(   (x_{g^{'}} ), \Phi_{g^{'}} (x_e)) <\varepsilon  $ and $\delta <\varepsilon /2$, follows $  d(x_g,\Phi_g (x_e))<\varepsilon$.\\

Since $\Phi_s|_{I_{s^{'}}}$ is a contraction,  $x_g$ and $\Phi_g (x_e)$ belong to the same interval $I_s$, $s\in \{ a,a^{-1},b,b^{-1}\}$.

\end{proof}

Final comments:\\
This example is easily generalizable to $S^{n}$. In the proofs made in this section, it was not used at all that $X=S^{1}$.\\
A dynamic system can have a Cantor set as a minimal set and not have the shadowing property. Just take a Denjoy homeomorphism $ \Phi_a $ and $ \Phi_b = Id $, and the action generated by $ \Phi_a $ and $ \Phi_b $. \\

\section{Proof of Theorem A}

We will start by getting some helpful results for our proof.

Let $(G,X,\Phi )$ be a dynamical system. We say that a point $N\in X$ is expansive point for $g\in G$ if there exists $\alpha >0$ such that for any $\delta >0$, if $0<d(y,N)<\delta$, then there exists $n\in \Z$ such that $d(\Phi_{g^{n}}(y), \Phi_{g^{n}}(N))>\alpha$. The number $\alpha$ is called a expansive constant of $N$.

\begin{rk}\label{rk3}
Let $N$ be an expansive point for $\Phi_a$ with expansive constat $\alpha$. Let $\{x_g\}$ be a $\delta$-pseudotrajectory with $x_{a^{m}}=\Phi_{a^{m}}(N)$ for all $m\in\Z$. If there exists $y\in X$ such that $d( x_g, \Phi_g(y))<\alpha$ for all $g\in G$, then $y=N$.
\end{rk}

\begin{lemma}\label{lema_general}

Let $(G,X,\Phi )$ be a dynamical system, $X$ a metric space, $N$ is a not isolated point and it is an expansive point for $\Phi_a$. If there exists an invariant connected set $M$  with $M\subset \overline{O(N)}$, then $\Phi$ does not have the shadowing property.
\end{lemma}
\begin{proof}

For each $s\in \{a,a^{-1},b,b^{-1} \}$ let $$F_2^{s}=\{ g\in F_2:\ g=sg_1 \}.$$ Note that $F_2= \bigcup _{s\in \{a,a^{-1},b,b^{-1} \} }F_2^{s}$.
Consider $F_2^{s}(N)=  \{\Phi_g(N), \ g\in F_2^{s} \} $. Since $M\subset \overline{O(N)}$, thus $M\subset \bigcup _{s\in \{a,a^{-1},b,b^{-1} \} }\overline{F_2^{s}(N)}$.\\

{\bf{Claim:}} $\overline{F_2^{s}(N)}\cap M\neq\emptyset$ for all $s\in \{a,a^{-1},b,b^{-1} \}$.\\
Since $M\subset \bigcup _{s\in \{a,a^{-1},b,b^{-1} \} }\overline{F_2^{s}(N)}$ then $M\cap \overline{F_2^{s_{0}}(N)}\neq\emptyset$ for some $s_0\in \{a,a^{-1},b,b^{-1} \}$. Then there exists $\{g_n \}\subset F_2^{s_{0}}$ and $z\in M$ such that $\Phi_{g_{n}}(N)\to z$. Hence, given $s\in \{a,a^{-1},b,b^{-1} \}$ with $s\neq s_0^{1}$ we have that $sg_n\in F_2^{s}$ and  $\Phi_s\Phi_{g_{n}}(N)\to \Phi_s(z)\in M$. Then $\overline{F_2^{s}(N)}\cap M\neq\emptyset$.\\
If $s=s_0^{-1}$, consider $s_1\neq s_0$ , $s_1\neq s^{-1}_0$. As proved above $\overline{F_2^{s_{1}}(N)}\cap M\neq\emptyset$. Reasoning in a similar way to the previous case, changing $s_0$ for $s_1$, we conclude $\overline{F_2^{s^{-1}_{0}}(N)}\cap M\neq\emptyset$.$\Box$\\

For $s \in \{a,a^{-1},b,b^{-1} \}$, let $A_s=\overline{F_2^{s}(N)}\cap M$. Note that $A_s$ is a non-empty closed set in $M$ for all $s$. Since $M$ is a connected set,

$$\left(\left(\overline{F_2^{b}(N)} \cup \overline{F_2^{b^{-1}}(N)} \right)     \cap M \right) \cap  \left(   \left(\overline{F_2^{a}(N)} \cup \overline{F_2^{a^{-1}}(N)} \right)     \cap M \right)  \neq\emptyset . $$

Without loss of generality we can suppose that  $\overline{F_2^{b}(N)}$ intersect to $\overline{F_2^{s}(N)}$ with $s\neq b$.\\

Suppose that $\Phi$ has the shadowing property for $\varepsilon , \delta$, with $\varepsilon <\alpha /2$ where $\alpha$ is  the expansive constant of $N$ for $\Phi_a$.

Since  $\overline{F_2^{b}(N)}\cap \overline{F_2^{s}(N)}\neq \emptyset$ then there exists $g_1, g_2\in G$  such that $$d (\Phi_{bg_{1}}(N), \Phi_{sg_{2}}(N))<\delta . $$

 Let $g_2=s_1...s_ra^{m} $ with $m\in\Z$ y $s_r\neq a$.

 Thus $$d (\Phi_{bg_{1}}(N), \Phi_{ss_1...s_r}(\Phi_{a^{m}}(N))<\delta . $$

%
%
%

let's  $g_1=s^{'}_1... s^{'}_t $ and $g_2=s_1...s_ra^{m} $ with $m\in\Z$ and $s_r\neq a$.
let's define the $\delta$-pseudotrajectory $\{x_g\}$ as follows(see figure \ref{figura7}):\\

\begin{itemize}

\item   $x_{a^{n}}=\Phi_{a^{n}}(N)$ for all $n\in\Z$,\\
\item $x_{ s^{'}_t }  =\Phi_{s^{'}_t }(N)$, $x_{s^{'}_{t-1} s^{'}_t }  =\Phi_{s^{'}_{t-1} s^{'}_t }(N),$..., $x_{ g_{1} }  =\Phi_{g_{1}}(N),$

\item  $x_{ bg_{1} }  =\Phi_{sg_{2}}(N)=\Phi_{ss_1...s_r}(\Phi_{a^{m}}(N)) ,$ ( here we use that $d (\Phi_{bg_{1}}(N), \Phi_{sg_{2}}(N))<\delta . $ )

\item $x_{ s^{-1}bg_{1} }  =\Phi_{g_{2}}(N)=\Phi_{s_1...s_r}(\Phi_{a^{m}}(N))$, ( here we use that $s\neq b$)

\item $x_{ s^{-1}_1s^{-1}bg_{1} }  =\Phi_{s_2...s_r}(  \Phi_{a^{m}}(N))$,..., $x_{ s^{-1}_{r-1}...s^{-1}_1s^{-1}bg_{1} }  = \Phi_{s_{r}a^{m}}(N)$ and

\item $x_{ s^{-1}_r...s^{-1}_1s^{-1}bg_{1} }  = N^{'}$ such that  $d(N^{'} , \Phi_{a^{m}}(N))<\delta$.

\end{itemize}

From here we define the pseudotrajectory dynamically.

Since $x_{a^{n}}=\Phi_{a^{n}}(N)$ for all $n\in\Z$, by Remark \ref{rk3}, if there exists $y\in X$ such that $d( x_g, \Phi_g(y))<\alpha$ for all $g\in G$ then $y=N$.

Since $\Phi_{s^{-1}_r...s^{-1}_1s^{-1} } (\Phi_{s{g_{2}}}(N)) =\Phi_{a^{m}}(N)$, then  $\Phi_{s^{-1}_r...s^{-1}_1s^{-1}bg_{1}}(N) \neq \Phi_{a^{m}}(N)$. Hence we take $N^{'} \neq \Phi_{a^{m}}(N)$


 Now ,applying by Remark \ref{rk3} to point $\Phi_{a^{m}}(N)$ we obtain a contradiction.

\end{proof}

Now we state a result about the classification of homeomorphisms with the shadowing property.

\begin{figure}[h]

\psfrag{sg2}{$\Phi_{sg_{2}}(N)$}

\psfrag{bg1}{$\Phi_{bg_{1}}(N)$}

\psfrag{an}{$\Phi_{a}(N)$}

\psfrag{amn}{$N^{'}$}
\psfrag{n}{$N$}

\psfrag{gi0}{$\Phi_b(I_{0})$}

\psfrag{gi1}{$\Phi_b(I_{1})$}

\psfrag{gj0}{$\Phi_b(I_{2})$}

\psfrag{gj1}{$\Phi_b(I_{3})$}

\psfrag{ggi0}{$\Phi_b^{-1}(I_{0})$}

\psfrag{ggi1}{$\Phi_b^{-1}(I_{1})$}

\psfrag{ggj0}{$\Phi_b^{-1}(I_{2})$}

\begin{center}
\caption{\label{figura7}}
\subfigure[]{\includegraphics[scale=0.3]{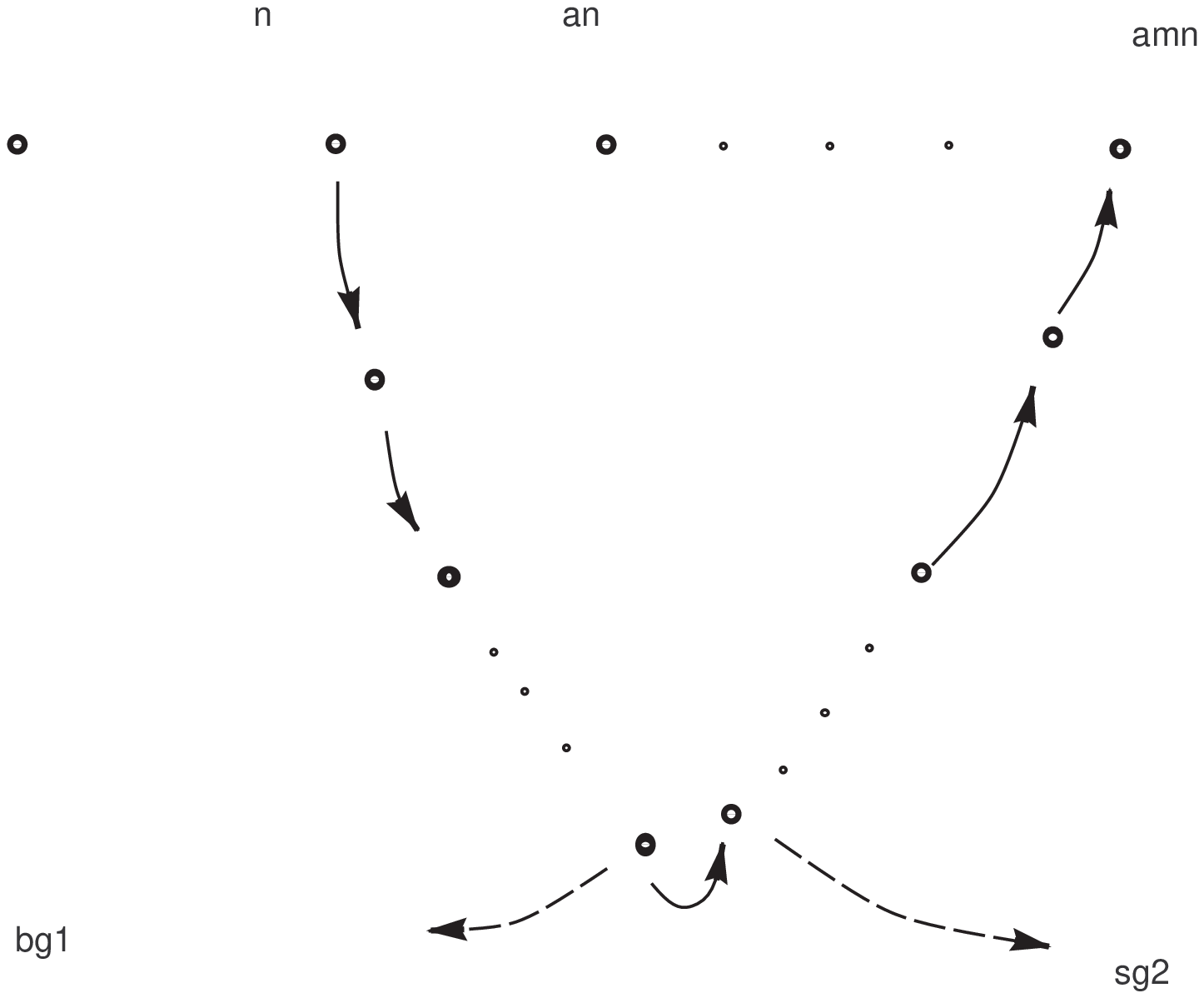}}
\end{center}
{ }
\end{figure}

Let a,b be two fixed point of a homeomorphism $f:S^{1}\to S^{1}$ preserving orientation with $(a,b)\cap Fix (f)=\emptyset$. The interval $(a,b)$ is an $r$-interval if for $x\in (a,b)$ we have that $f^{n}(x)\to a$, $n\to -\infty$ and  $f^{n}(x)\to b$, $n\to +\infty$.
The interval $(a,b)$ is an $l$-interval if for $x\in (a,b)$ we have that $f^{n}(x)\to a$, $n\to +\infty$ and  $f^{n}(x)\to b$, $n\to -\infty$.
Let $[f]=\{f^{m}: \ m\in \N \}$.

The following result is due to Plamenevskaya (see \cite{pl}).

\begin{theorem}\label{teorema_shadowing}
  A homeomorphism $f:S^{1}\to S^{1}$ has the shadowing property if and only if the family $[f]$ contains a homeomorphism such that:
  \begin{enumerate}
    \item $f$ preserve orientation.
    \item The set $Fix(f)$ is nowhere dense and contains at least two points.
    \item For any two $r$-intervals ($l$-intervals) $(a,b)$ and $(c,d)$ there exists $l$-intervals ( correspondingly, $r$-intervals) $(p,q)$ and $(s,t)$ such that $$ (p,q)\subset (b,c) \mbox{  and  } (s,t)\subset (d,a).$$
  \end{enumerate}
\end{theorem}
%

{\bf{ The set $S^{1}$ is the minimal set.}}\\

We will now prove the Theorem A for the case that $S^{1}$ is the minimal set.

By Theorem \ref{teorema_shadowing}, there exist $m,n\in\Z$ such that the maps $\Phi_{a^{m}}$ and $\Phi_{b^{n}}$ preserve orientation and they have at least two fixed points.

Changing, if necessary, the free group $F_2$ by the free group $F_2^{'}$ generated by $\{a^{m},b^{n}\}$, we can assume that the maps $\Phi_{a}$ and $\Phi_{b}$  preserve orientation and they have at least two fixed points.

Again, by Theorem \ref{teorema_shadowing} there exists a $r$-interval or a $l$-interval $(N,N^{'})$ for $\Phi_a$. Suppose that $(N,N^{'})$ is a $r$-interval.

By contradiction, suppose that $\Phi$ has the shadowing property for $\varepsilon$, $\delta$.
Since $S^{1}$ is a connected invariant set and $\overline{O(N)}=S^{1}$, by the proof of Lemma \ref{lema_general} there exist $g_1, g_2\in G$  such that $$d (\Phi_{bg_{1}}(N), \Phi_{sg_{2}}(N))<\delta \mbox{ with  } s\neq b . $$

Suppose that $\Phi_{bg_{1}}(N)$ is on the left of $\Phi_{sg_{2}}(N)$.

Let's write $g_1=s^{'}_1... s^{'}_t $ and $g_2=s_1...s_ra^{m} $ with $m\in\Z$ and $s_r\neq a$.
Let's define a $\delta$-pseudotrajectory $ \{x_g \} $ as follows:

\begin{itemize}

\item   $x_{a^{n}}=\Phi_{a^{n}}(N)$ for all $n\in\Z$,\\
\item $x_{ s^{'}_t }  =\Phi_{s^{'}_t }(N)$, $x_{s^{'}_{t-1} s^{'}_t }  =\Phi_{s^{'}_{t-1} s^{'}_t }(N),$..., $x_{ g_{1} }  =\Phi_{g_{1}}(N),$

 \item  $x_{ bg_{1} }  =\Phi_{sg_{2}}(N)=\Phi_{ss_1...s_r}(\Phi_{a^{m}}(N)) ,$ (here we use that $d (\Phi_{bg_{1}}(N), \Phi_{sg_{2}}(N))<\delta . $ )

\item $x_{ s^{-1}bg_{1} }  =\Phi_{g_{2}}(N)=\Phi_{s_1...s_r}(\Phi_{a^{m}}(N))$, ( here we use that $s\neq b$)

\item $x_{ s^{-1}_1s^{-1}bg_{1} }  =\Phi_{s_2...s_r}(  \Phi_{a^{m}}(N))$,..., $x_{ s^{-1}_{r-1}...s^{-1}_1s^{-1}bg_{1} }  = \Phi_{s_{r}a^{m}}(N)$ and

\item $x_{ s^{-1}_r...s^{-1}_1s^{-1}bg_{1} }  = y $ such that $y\in (N,N^{'})$ and  $d(y ,N)<\delta$.

\end{itemize}

From here we define the pseudotrajectory dynamically.

Suppose that there exist $z\in S^{1}$ be such that $d( x_g, \Phi_g(z))<\varepsilon$ for all $g\in G$.

Since  $x_{a^{n}}=\Phi_{a^{n}}(N)$ for all $n\in\Z$ and $(N,N^{'})$ is a $r$-interval, then $z\in (N-\varepsilon, N)$. (see figure \ref{figura8}).

Since $\Phi_g$ preserve orientation and $z$ is on the left of $N$, then $\Phi_{bg_{1}}(z)$ is on the left of $\Phi_{bg_{1}}(N)$. Recall that  $\Phi_{bg_{1}}(N)$ is on the left of $\Phi_{sg_{2}}(N)$.

Then $\Phi_{g^{-1}_2s^{-1}bg_{1}}(z)$ is on the left of $\Phi_{g^{-1}_2s^{-1}}(\Phi_{sg_{2}}(N))= N$.

Since $\Phi_{a^{n}}  (   \Phi_{g^{-1}_2s^{-1}bg_{1}}(z))\notin   (N, N^{'})$ for all $n\geq 0$  y $\Phi_{a^{n}}  ( y)\to       N^{'}$ because $(N,N^{'})$ is $r$-interval we obtain a contradiction.

\begin{figure}[h]

\psfrag{sg2}{$\Phi_{sg_{2}}(N)$}

\psfrag{bg1}{$\Phi_{bg_{1}}(N)$}
\psfrag{bbg1}{$\Phi_{bg_{1}}(z)$}
\psfrag{an}{$\Phi_{a}(N)$}

\psfrag{amn}{$N^{'}$}
\psfrag{n}{$N$}

\psfrag{nn}{$N^{'}$}
\psfrag{z}{$z$}

\begin{center}
\caption{\label{figura8}}
\subfigure[]{\includegraphics[scale=0.17]{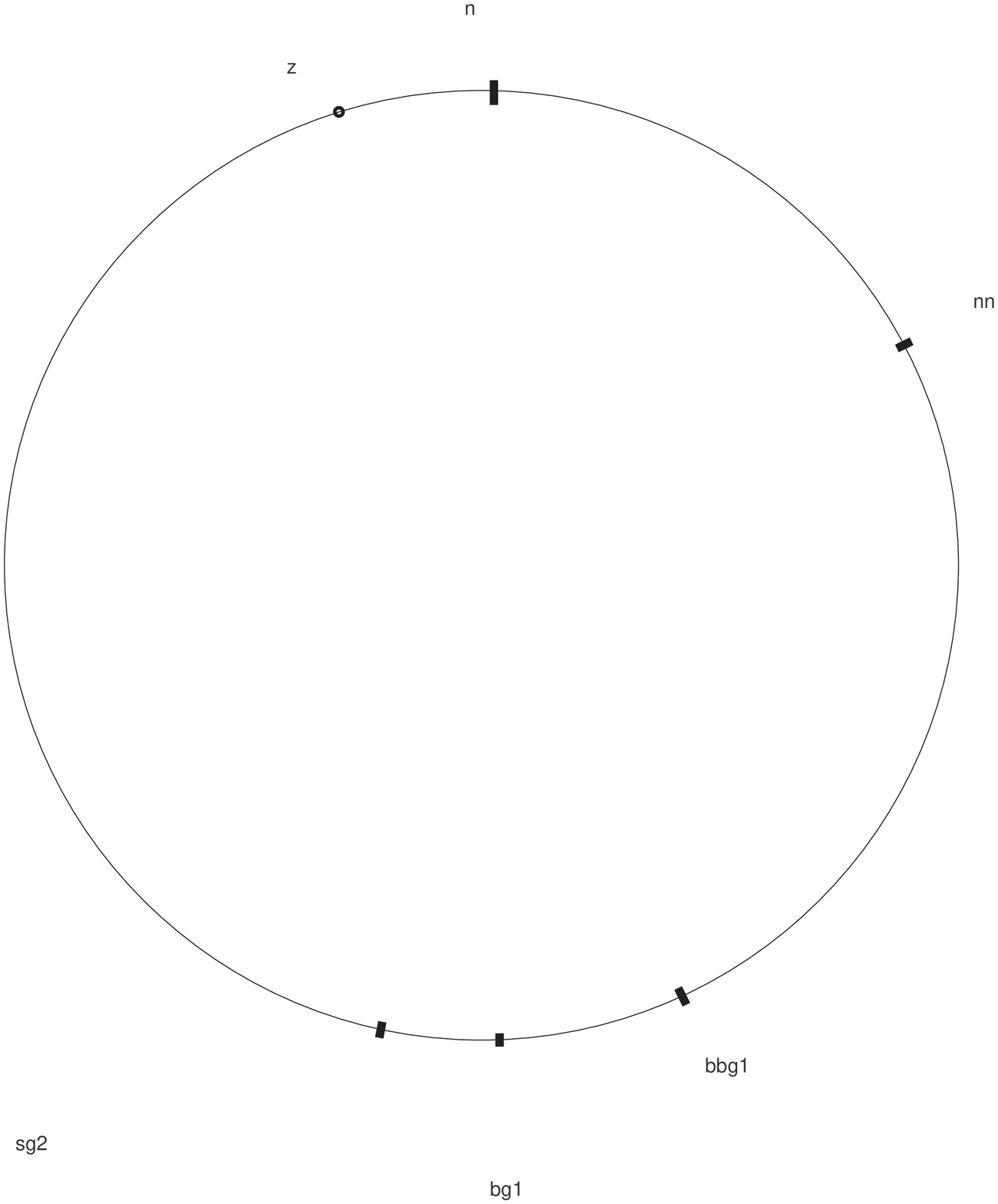}}
\end{center}
{ }
\end{figure}

{\bf{There exists a finite minimal set.}}\\
First we state a lemma and the Theorem A is obtained as corollary.

\begin{lemma}\label{libre}
Let $(F_2,S^{1},\Phi )$ be a dynamical system. Suppose that there exist $g_1,g_2\in F_2$ such that:
\begin{enumerate}
 \item $g^{n}_1\neq g^{m}_2$ for all $n,m\in\Z\setminus \{0\}$ and
 \item there exists $N\in S^{1}$ such that $\Phi_{g_{1}}(N)=\Phi_{g_{2}}(N)=N$.
\end{enumerate}
Then $\Phi$ does not have the shadowing property.
\end{lemma}

\begin{proof}
 Let $F_2^{'}$ be the subgroup of $F_2$ generated by $g_1,g_2$. We consider the action $\Phi^{'}=\Phi_|{_{F_2^{'}}}$. Note that if the dynamical system $(F^{'}_2,S^{1},\Phi^{'} )$  has not have the shadowing property then $(F_2,S^{1},\Phi )$  has not have the shadowing property.

Since $g^{n}_1\neq g^{m}_2$ for all $n,m\in\Z\setminus \{0\}$ the group $F_2^{'}$ is isomorphic with $F_2$.

 To simplify we write $\Phi^{'}_{a}=\Phi_{g_{1}}$ and $\Phi^{'}_{b}=\Phi_{g_{2}}$.\\
 We will prove that the lift of the dynamical system $(F^{'}_2,S^{1},\Phi^{'} )$  has not have the shadowing property. Thus the dynamical system $(F^{'}_2,S^{1},\Phi^{'} )$  has not have the shadowing property either.

Let $(\Pi , \R )$ be the universal covering of $S^{1}$ with $\Pi: \R \to S^{1}$ such that $\Pi (t)=e^{i2\pi t}$. Let $\widetilde{N}\in \R$ be such that $\Pi(\widetilde{N})=N$ and $\widetilde{\Phi}_a ,\widetilde{\Phi}_b :\R\to\R$, the lift of $\widetilde{\Phi}_a$ and $\widetilde{\Phi}_a$ that fix  $\widetilde{N}$.

Let $(F^{'}_2,\R,\widetilde{\Phi} )$ be the dynamical system where $\widetilde{\Phi}$  is the action generated by  $\widetilde{\Phi}_a$  and  $\widetilde{\Phi}_b$.
Note that $\Phi^{'}$ has the shadowing property iff  $\widetilde{\Phi}$ has the shadowing property.

We will prove that the action $\widetilde{\Phi}$ has not have the shadowing property.

Note that the interval $[ \widetilde{N}, \widetilde{N}+1]$ is invariant for the action $\widetilde{\Phi}$. Then for all $x\in [ \widetilde{N}, \widetilde{N}+1]$, $\overline{O(x)}\subset [ \widetilde{N}, \widetilde{N}+1]$.

Suppose that the action $\widetilde{\Phi}$ has the  shadowing property for $\varepsilon, \delta$.

Recall that $\widetilde{\Phi}_a$ and $\widetilde{\Phi}_b$ are increasing functions.

Given $x\in [ \widetilde{N}, \widetilde{N}+1]$, consider a $\delta$-pseudotrajectory as follows:

$x_e=x$.

We have two possibilities: $\widetilde{\Phi}_a(x_e)\geq x_e$ or $\widetilde{\Phi}_{a^{-1}}(x_e)\geq x_e$. Hence there exists $s_1\in \{a,a^{-1} \}$ such that $\widetilde{\Phi}_{s_{1}}(x_e)\geq x_e$
Thus, define  $$x_{s_{1}}     =   \widetilde{\Phi}_{s_{1}}(x_e)+\frac{\delta}{2} \geq x_e +\frac{\delta}{2}\cdot$$

Now consider $x_{s_{1}}$.
We have two possibilities: $\widetilde{\Phi}_b(x_{s_{1}})\geq x_{s_{1}}$ or $\widetilde{\Phi}_{b^{-1}}(x_{s_{1}})\geq x_{s_{1}}$. Hence there exists $s_2\in \{b,b^{-1} \}$ such that $\widetilde{\Phi}_{s_{2}}(x_{s_{1}})\geq x_{s_{1}}$
Thus, define $$x_{ {s_{2}}s_{1}}     =   \widetilde{\Phi}_{s_{2}}(x_{s_{1}})+\frac{\delta}{2}\geq  x_{s_{1}}   +\frac{\delta}{2} \geq x_e +\delta\cdot  $$
Again, consider $x_{ {s_{2}}s_{1}} $ and the maps $\widetilde{\Phi}_a $ and $ \widetilde{\Phi}_{a^{-1}}$ to define the point $x_{ {s_{3}}{s_{2}}s_{1}} $. Clearly
$$x_{ {s_{n}}...{s_{2}}s_{1}}  \underset{n\to +\infty}{\longrightarrow} +\infty .$$  Since $\overline{O(x)}\subset [ \widetilde{N}, \widetilde{N}+1]$, then $\widetilde{\Phi}$ does not have the shadowing property.

\end{proof}

\begin{coro}\label{final}
 Let $(F_2,S^{1},\Phi )$ be a dynamical system. If there exists a finite minimal set, then $\Phi$ does not have the shadowing property.
\end{coro}
\begin{proof}
Let $M$ be a finite minimal set for $(F_2,S^{1},\Phi )$. Without loss of generality we can assume that $ M = \{N \} $ and $\Phi_g(N)=N$ for all $g\in F_2$. Thus, by Lemma \ref{libre} $\Phi$ does not have the shadowing property.
\end{proof}

To end we give a simple example that show that Lemma \ref{libre} and Corollary \ref{final} are not equivalent. Let $f:S^1\to S^1$ be a irrational rotation. Consider  $\Phi_a=f$, $\Phi_b=f^{-1}$ and $\Phi$ the action generated by $\Phi_a$ and $\Phi_b$. Let $g_1=ab$ and $g_2=ba$.
Thus $g^n_1\neq g_2^m$ for all $m,n\in \Z\setminus \{0\}$ and $\Phi_{g_{1}}=\Phi_{g_{2}}=Id$. Hence for all $N\in   S^1$,  $\Phi_{g_{1}}(N)=\Phi_{g_{2}}(N)=N$ and  $S^1$ is the minimal set for the action $\Phi$.


\begin{thebibliography}{99}

%
%
%
%


\bibitem[N]{n} (MR2394157)
\newblock A. Navas,
\newblock \emph{Groups of circle diffeomorphisms. Translation of the 2007 Spanish edition.},
\newblock  Chicago Lectures in Mathematics. University of Chicago Press, Chicago, IL, 2011. xviii+290 pp. ISBN: 978-0-226-56951-2; 0-226-56951-9.



\bibitem[OT]{ot} (MR1944405) [10.1017/S0143385702000512]
\newblock Osipov, A., Tikhomirov , S,
\newblock \doititle{Shadowing for actions of some finitely
generated groups.  },
\newblock \emph{ Dyn. Syst.,} \textbf{29}, no. 3, (2014), 337--351.


\bibitem[P]{p} (MR1944405) [10.1017/S0143385702000512]
\newblock Pilyugin, S,
\newblock \doititle{Theory of pseudo-orbit shadowing in dynamical systems.  },
\newblock \emph{ Differ. Equ.,} \textbf{47}, no. 13, (2011), 1929--1938.


\bibitem[Pl]{pl} (MR1944405) [10.1017/S0143385702000512]
\newblock Plamenevskaya, O. B.,
\newblock \doititle{Pseudo-orbit tracing property and limit shadowing property on a circle.  },
\newblock \emph{ Vestnik St. Petersburg Univ. Math,} \textbf{30}, no. 1, (1997), 27--30.








\end{thebibliography}
\end{document}